\documentclass[11pt,a4paper]{amsart}
\usepackage[utf8]{inputenc}

\title{Tiling in some nonpositively curved groups}
\author{Joseph MacManus and Lawk Mineh}
\date{First draft: 17 January 2024. This version: 19 Feburary 2026.}

\address{School of Mathematics, University of Bristol, Bristol, BS8 1UG, UK}
\email{joseph.macmanus@bristol.ac.uk}
\address{Mathematical Institute, University of Bonn, D-53115 Bonn, Germany}
\email{lawk@math.uni-bonn.de}

\usepackage{import}
\input{preamble.sty}


\begin{document}

\begin{abstract}
We prove that acylindrically hyperbolic groups are monotileable. That is, every finite subset of the group is contained in a finite tile. This provides many new examples of monotileable groups, and progress on the question of whether every group is monotileable. In particular, one-relator groups and many Artin groups are monotileable.
\end{abstract}

\maketitle

\section{Introduction}

A \textit{tile} of a group $G$ is a subset $T \subset G$ such that $G$ can be covered by a disjoint union of translates of $T$. A group $G$ is said to be \textit{monotileable} if every finite subset is contained in a finite tile. 
A long-standing question, asked by Chou \cite{Chou} and Weiss \cite{Weiss}, is the following. 

\begin{question*}
    Is every group monotileable?
\end{question*}

Chou and Weiss independently introduced monotileability with a view towards applications to ergodic theory, via Rohklin sets in amenable groups.
In previous literature, the terms \emph{monotileable}, \emph{MT}, and \emph{Property (P)} have variously been used to refer to the same property.
Chou showed that being monotileable is preserved under extensions, directed unions, being fully residually monotileable, and under free products \cite{Chou}.
Moreover, the class of monotileable groups straightforwardly includes finite groups and free abelian groups.
It follows that elementary amenable groups are monotileable, while it remains open whether all amenable groups are monotileable. 
Note that since monotileability is closed under directed unions, the above question reduces to the class of finitely generated groups.

In \cite{Seward_PT}, Seward explores various related tiling properties of groups, including the so-called `CCC' property (standing for ``coherent, centered, cofinal''), which is a stronger, uniform version of being monotileable.
Gao, Jackson, and Seward show that many classes of groups are CCC, including non-trivial countable free products of countable groups \cite[Theorem 4.5.7]{Gao_Jackson_Seward}.
Whether all groups have the CCC property also remains open.
We would also like to mention a result of Akhmedov and Fulghesu that any subset  of a free group that is connected  with respect to the standard generating set forms a tile \cite[Proposition 5.1]{Akhemdoc_Fulghesu}.

\medskip

A finitely generated group is called \emph{word-hyperbolic} if it admits a proper cocompact action on a proper hyperbolic metric space (see Section~\ref{sec:prelims} for definitions).
Word-hyperbolic groups were introduced by Gromov in his seminal essay \cite{Gromov_hyperbolic}, unifying geometric and combinatorial methods in group theory; fundamental groups of closed hyperbolic manifolds and small cancellation quotients of free groups are archetypal examples.
In a recent paper \cite{akhmedov2023big}, Akhmedov proved the following. 

\begin{theorem*}[Akhmedov]
    Word-hyperbolic groups are monotileable. 
\end{theorem*} 

There are many groups admitting interesting actions on hyperbolic metric spaces which are not word-hyperbolic in the sense of Gromov.
One such class that has received considerable attention is that of \emph{acylindrically hyperbolic groups}.
Motivating examples of non-hyperbolic acylindrically hyperbolic groups include the mapping class groups \(\operatorname{MCG}(\Sigma)\) of non-exceptional surfaces \(\Sigma\) and outer automorphism groups of finitely generated free groups \(\operatorname{Out}(F_n)\).
The study of acylindrically hyperbolic groups was initiated by Osin in \cite{Osin_Acyl_hyp}; we refer the reader to \cite{Osin_survey} for a survey of examples and prominent features of the theory.

The main aim of this article is to extend the tiling method developed in \cite{akhmedov2023big} to the setting of acylindrically hyperbolic groups.

\begin{alphtheorem}\label{thmA}
    Acylindrically hyperbolic groups are monotileable. 
\end{alphtheorem}

It is a well-known open question whether all hyperbolic groups are residually finite, while there are examples of freely irreducible non-residually finite acylindrically hyperbolic groups (see, for example, the graphical small cancellation groups \(G(S)\) in \cite{Brown-thesis}, with \(S\) periodic).
Thus the above result is in any case strictly stronger than those previously known.

Recall that a \textit{one-relator group} is a group admitting a finite presentation with a single relator. 
By a result of Minasyan and Osin \cite{Minasyan_Osin}, many one-relator groups are known to be acylindrically hyperbolic, and the structure of non-acylindrically hyperbolic one-relator groups is somewhat constrained.
Hence we are able to use Theorem~\ref{thmA} to deduce the following. 

\begin{alphcorollary} \label{cor:one_relator}
    One-relator groups are monotileable. 
\end{alphcorollary}

Let \(\Gamma\) be a simple graph with integer edge labels \(m_{ij} \geq 2\) for each \(\{i, j\} \in E\Gamma\).
The associated \emph{Artin group} is the group with the presentation 
\[
    A_\Gamma = \Big\langle \, V\Gamma \, \Big| \, \overbrace{aba\dots}^{m_{ij}} = \overbrace{bab\dots}^{m_{ij}},\; \{i,j\} \in E\Gamma \, \Big\rangle
\]
Many Artin groups are known to be acylindrically hyperbolic, such as those of Euclidean type \cite{Calvez_Euclidean} and those whose graph \(\Gamma\) is not a join of two subgraphs \cite{Charney_MW}.
For these groups Theorem~\ref{thmA} applies straightforwardly.
We record an application to the well-known family of \emph{two-dimensional} Artin groups that is not entirely immediate.

\begin{alphcorollary}
\label{cor:artin}
    Two-dimensional Artin groups are monotileable.
\end{alphcorollary}

We would also like to note that it is conjectured to be the case that the central quotient \(A_\Gamma/Z(A_\Gamma)\) of any irreducible Artin group \(A_\Gamma\) is acylindrically hyperbolic.
If this were true, Theorem~\ref{thmA} would show that all Artin groups are monotileable, as they would be products of extensions of acylindrically hyperbolic groups by abelian groups.
However, this conjecture is likely quite difficult.

\medskip

We will briefly outline the argument behind the proof of Theorem~\ref{thmA}.
Let \(G\) be a group with Cayley graph \(\Gamma\), and let \(F \subset G\) be an arbitrary finite subset.
The main idea is to find an element \(z \in G\) with very large translation length that in some sense does not interact with \(F\).
In particular, multiplying elements of \(F\) on the left and right by powers of \(z\) should give one elements far away from \(F\) in \(\Gamma\) (with respect to the edge-path metric).

We show that the existence of such elements, which we call \emph{swingers}, together with \(\Gamma\) being hyperbolic, implies that \(G\) is monotileable.
In fact, taking \(z\) to be a sufficiently large swinger will make \(F \cup \{z\}\) into a tile.
One may think of \(z\) as `swinging' the set \(F\) around by large enough distances, so that \(z\) can plug the gaps in \(G\) and work to tile \(G\).
This idea is due wholly to Akhmedov, who presented the argument in the setting of hyperbolic groups in \cite{akhmedov2023big}.

Acylindrically hyperbolic groups have hyperbolic Cayley graphs, which have an ideal boundary at infinity on which the group acts.
We exploit the dynamics of this action on the boundary to find swingers in acylindrically hyperbolic groups.
The aim here is essentially to find \(z\) that acts by translation along an axis sufficiently far from fixed subspaces corresponding to a large ball about the identity in \(\Gamma\).
We recast this as a statement about fixed points of the action of \(G\) on the boundary of \(\Gamma\).
The main difficulty here comes from the fact that \(\Gamma\) may not be locally finite.

\subsection*{Acknowledgements} The authors would like to thank Azer Akhmedov for useful correspondence, Marco Linton for helpful discussions on one-relator groups, and Ashot Minasyan for careful reading of a draft of this paper. We also thank the anonymous referee for their detailed feedback.


\section{Preliminaries}
\label{sec:prelims}

We begin with a basic statement about tiling sets with two elements.

\begin{lemma}
\label{lem:two_element_tile}
    Let \(G\) be a group, \(A\) a \(G\)-set, and \(B \subset A\).
    Suppose there is \(g \in G\) such that \(B g \subset B\).
    If the order of \(g\) is even or infinite, \(B\) is a disjoint union of subsets of the form \(\{b,bg\}\), where \(b \in B\). 
\end{lemma}

\begin{proof}
    Let \(S = \{g^n : n \geq 0\}\) be the submonoid generated by \(g\).
    There is a subset \(C \subset B\) such that \(B = \bigsqcup_{c \in C} c S\).
    The claim thus follows by observing that \(S = \bigsqcup_{n \in \NN}\{1,g\} g^{2n}\) when \(g\) has even or infinite order. 
\end{proof}

\subsection{Quasiisometries and quasigeodesics}
Let $X$ and $Y$ be metric spaces. Given $\lambda \geq 1$, $c \geq 0$, we say that $f \colon X \to Y$ is a \textit{quasiisometric embedding} if
\[
\frac{1}{\lambda}\dist(x,y) - c \leq \dist(f(x), f(y)) \leq \lambda \dist(x,y) + c,
\]
for all $x,y \in X$.
We say that $f$ is \textit{$c$-coarsely surjective} if for every $y \in Y$ there exists $x \in X$ such that $\dist(y, f(x)) \leq c$. If $f$ satisfies both of the above, we call it a \textit{$(\lambda, c)$-quasiisometry}.  
If a map $p \colon [a,b] \to X$ is a $(\lambda,c)$-quasiisometric embedding, we call it a \textit{$(\lambda, c)$-quasigeodesic.} Similarly, a quasiisometric embedding $p \colon [0, \infty) \to X$ is called a \textit{quasigeodesic ray}. We will often abuse notation and identify a quasigeodesic (ray) as above with its image in $X$. 

\subsection{Hyperbolic spaces}

We will recall some basic definitions and facts about hyperbolic spaces and groups acting on them.

A triangle \(T\) in geodesic space \(X\) is said to be \emph{\(\delta\)-slim} if there is \(\delta \geq 0\) such that each side of \(T\) is contained in a \(\delta\)-neighbourhood of the other two sides
A geodesic space \(X\) is said to be \(\delta\)-\emph{hyperbolic} if each geodesic triangle in \(X\) is \(\delta\)-slim.

Two quasigeodesic rays in a \(\delta\)-hyperbolic space \(X\) are called \emph{asymptotic} if they are within finite Hausdorff distance of one another.
The relation of being asymptotic is an equivalence relation.
We call the set of equivalence classes of asymptotic \((1,20\delta)\)-quasigeodesic rays \(\partial X\) the \emph{boundary} of \(X\), and it carries a natural topology under which \(X \cup \partial X\) is a completely metrisable Hausdorff space containing \(X\) as a dense subset.
We say that a ray \(\ell\) in \(X\) \emph{converges to} \(p \in \partial X\) if \(\ell \in p\).
For any \(x,y \in X \cup \partial X\), there is a \((1,20\delta)\)-quasigeodesic with endpoints \(x\) and \(y\) \cite[Remark 2.16]{Benakli_Kapovich}.

If \(G\) acts by isometries \(X\), this induces an action by homeomorphisms on \(\partial X\).
The set \(\Lambda G\) of accumulation points of orbits of this action is called the \emph{limit set} of \(G\).
The limit set of \(G\) has either zero, one, two, or infinitely many points. 
In the first three cases, the action is called \emph{elementary}; in the last, \emph{nonelementary}.
Note that when the action of \(G\) is nonelementary, \(\Lambda G\) has no isolated points \cite[Theorems 2.10, 4.5]{Hamann}.
If the action of \(G\) on \(X\) is cobounded, then it is straightforward to see that \(\Lambda G = \partial X\).

For ease of reference, we record the following two standard facts. The following is commonly referred to as the \textit{Morse lemma}. 

\begin{lemma}[{\cite[Theorem~11.72, Lemma~11.75]{dructu2018geometric}}]\label{lem:morse}
    Let $X$ be a hyperbolic geodesic metric space. 
    For every $\lambda \geq 1$, $c \geq 0$, there exists a constant $\mu = \mu(\lambda, c) \geq 0$ such that the following holds. 
    
    If $p$ is a geodesic with endpoints $a, b \in X \cup \partial X$ and $q$ is a $(\lambda, c)$-quasigeodesic with the same endpoints as $p$, then the Hausdorff distance between $p$ and $q$ is bounded above by $\mu$. 
\end{lemma}

The constant $\mu(\lambda, c)$ in Lemma~\ref{lem:morse} is sometimes referred to as the \textit{Morse constant} for $(\lambda, c)$-quasigeodesics. 

For the next statement we need an extra definition. Let $x_0, x_1, x_2 \in X \cup \partial X$ be distinct points. Then a \textit{generalised triangle} with vertices $x_0$, $x_1$, $x_2$ is a union of three, possibly infinite, geodesics $p_0$, $p_1$, $p_2$ such that the endpoints $p_i$ are $x_i$ and $x_{i+1}$, where indices are taken modulo 3. More generally, if we relax the $p_i$ to be $(\lambda, c)$-quasigeodesic triangles, then we call the resulting figure a \textit{$(\lambda, c)$-quasigeodesic generalised triangle}.

\begin{lemma}[{\cite[Exercise~11.86]{dructu2018geometric}}]\label{lem:slim-ideal-geotriangles}
    Let $X$ be a $\delta$-hyperbolic geodesic metric space. Then every generalised triangle is $5\delta$-slim. 
\end{lemma}

Combining Lemmas~\ref{lem:morse} and \ref{lem:slim-ideal-geotriangles}, one immediately deduces the following. 

\begin{lemma}\label{lem:slim-ideal-triangles}
    Let $X$ be a hyperbolic geodesic metric space. For every $\lambda \geq 1$, $c \geq 0$, there exists a constant $\delta' = \delta'(\lambda, c) \geq 0$ such that every $(\lambda, c)$-quasi-geodesic generalised triangle is $\delta'$-slim. 
\end{lemma}

\subsection{Acylindrical actions}

Let \(G\) be a group acting by isometries on metric space \(X\).
The action of \(G\) is called \emph{acylindrical} if for every \(\varepsilon > 0\) there are constants \(R, N \geq 0\) such that the set
\[
    \{g \in G \, | \, \dist(x,gx) < \varepsilon, \dist(y,gy) < \varepsilon\}
\]
has at most \(N\) elements whenever \(x,y \in X\) are such that \(\dist(x,y) \geq R\). 

An element \(g \in G\) is called \emph{elliptic} if it acts on \(X\) with bounded orbits.
An element \(g \in G\) is called \emph{loxodromic} if the map \(\ZZ \to X\) given by \(n \mapsto g^n x\) is a quasiisometric embedding for any \(x \in X\).
If \(X\) is hyperbolic, a loxodromic element \(g \in G\) fixes exactly two points of \(\partial X\), which we write as \(g^\infty\) and \(g^{-\infty}\).
Two loxodromic elements \(g, h \in G\) are said to be \emph{independent} if the sets \(\{g^\infty, g^{-\infty}\}\) and \(\{h^\infty, h^{-\infty}\}\) are disjoint.

Acylindrical actions and their individual isometries have been classified.
If \(G\) acts acylindrically and the space \(X\) is hyperbolic, then every element is either elliptic or loxodromic \cite[Lemma 2.2]{Bowditch_Tight_Geods}.
The actions are known to fall into the trichotomy described below.

\begin{theorem}[{\cite[Theorem 1.1]{Osin_Acyl_hyp}}]
\label{thm:osin_trichotomy}
    Let \(G\) be a group acting acylindrically on a hyperbolic space. Then either
    \begin{itemize}
        \item \(G\) has bounded orbits,
        \item \(G\) is virtually cyclic and contains a loxodromic element, or
        \item \(G\) contains infinitely many independent loxodromic elements.
    \end{itemize}
\end{theorem}

A group $G$ is said to be \textit{acylindrically hyperbolic} if it admits a nonelementary acylindrical action on a hyperbolic geodesic metric space $X$. If such a space $X$ exists, then there exists a generating set $S$ of $G$ such that the Cayley graph $\Gamma(G,S)$ is hyperbolic and the action of $G$ on $\Gamma(G,S)$ is nonelementary and acylindrical \cite[Theorem~1.2]{Osin_Acyl_hyp}. 

We will need the following fact about existence of loxodromic elements with particular endpoints.

\begin{theorem}[{\cite[Theorems~2.9, 4.5]{Hamann}}]
\label{thm:lox_fixpoints_dense}
    Let \(G\) be a group acting acylindrically and nonelementarily on hyperbolic space \(X\).
    Then for any pair of disjoint nonempty open sets \(U, V \subset \Lambda G\), there is a loxodromic element \(g \in G\) with \(g^\infty \in U\) and \(g^{-\infty} \in V\).
\end{theorem}


\section{Fixed points at infinity}
\label{sec:fixed_points}

Throughout this section we take $X$ to be a \(\delta\)-hyperbolic geodesic metric space and \(G\) a group with an acylindrical, nonelementary action on \(X\) by isometries. 
The goal of this section is to study fixed subsets of \(\partial X\) under the induced action of \(G\) by homeomorphisms. 
We introduce the following notation. Given $g \in G$, we write 
\[
    \Fix(g) = \{p \in \partial X : g\cdot p = p\}.
\]

The following two lemmas are well-known, but to the knowledge of the authors do not appear in the literature.

\begin{lemma}
\label{lem:lox_fixpoints_disjoint}
    Let \(g, h \in G\) be loxodromic elements. Then \(\Fix(g)\) and \(\Fix(h)\) are either disjoint or equal.
\end{lemma}

\begin{proof}
    Suppose that \(g\) and \(h\) share at least one fixed point \(p \in \partial X\).
    By Theorem~\ref{thm:osin_trichotomy}, the subgroup \(\langle g, h \rangle\) is either virtually cyclic or contains infinitely many independent loxodromic elements.
    In the former case, \(g\) and \(h\) are commensurate, and so \(\Fix(g) = \Fix(h)\).
    In the latter case, \(\langle g, h \rangle\) contains a loxodromic element not fixing \(p\), which is a contradiction.
\end{proof}

\begin{lemma}
\label{lem:pair_stab}
    Let $p, q \in \partial X$ be distinct. Then the stabiliser of $\{p, q\}$ is either virtually cyclic and contains a loxodromic element, or is finite. 
\end{lemma}

\begin{proof}
    Denote by \(N \leqslant G\) the stabiliser of \(\{p,q\}\).
    Passing to an index two subgroup if necessary, we may assume that \(N\) fixes both \(p\) and \(q\). 
    Let \(\ell\) be a bi-infinite quasigeodesic with endpoints \(p\) and \(q\).
    As \(N\) fixes \(p\) and \(q\), \(s \cdot \ell\) is quasigeodesic lying a uniform finite Hausdorff distance from \(\ell\) for any \(s \in N\).
    
    By Theorem~\ref{thm:osin_trichotomy}, \(N\) acts with bounded orbits, is virtually cyclic containing a loxodromic, or contains independent loxodromic elements.
    The final case may not occur, as \(N\) would contain a loxodromic element fixing neither \(p\) nor \(q\).

    Suppose, then, that \(N\) is not virtually cyclic containing a loxodromic.
    Then the orbits of \(N\) are bounded in \(X\).
    It follows that the diameter of the orbit \(Nx\) is uniformly bounded for any \(x \in \ell\).
    Picking \(x, y \in \ell\) sufficiently far apart, the definition of acylindricity implies that \(N\) is finite.
\end{proof}

Since the action of \(G\) is nonelementary, \(\partial X\) contains more than two points.
The following is then a straightforward consequence of the above lemma.

\begin{corollary}
\label{cor:acyl_boundary_stab_finite}
    The pointwise stabiliser \(N = \{g \in G : \Fix(g) = \partial X\}\) of the boundary is a finite normal subgroup of $G$. 
\end{corollary}

The following statement will also be helpful.

\begin{lemma}\label{lem:lox-product}
    Let $g\in G$ be loxodromic and let \(h \in G\) be such that $h \cdot g^\infty \neq g^{-\infty}$. 
    There exists $N \in \NN$ such that the element $h g^n $ is also loxodromic for all \(n \geq N\). In particular, $hg^n$ has infinite order.
\end{lemma}

\begin{proof}
    Write $\widehat X = X \cup \partial X$ for the completion of $X$ with its boundary. 
    Let $U$ and $V$ be open neighbourhoods in \(\widehat X\) of $g^\infty$ and $g^{-\infty}$ respectively such that $\overline U \cap \overline V = \emptyset$, $\overline U \cup \overline V \neq \widehat X$, and \(h \overline{U} \cap \overline V = \emptyset\). 
    
    We have that \(hU\) and \(V\) are open neighbourhoods of \(h \cdot g^\infty\) and \(g^{-\infty}\) with \(\overline{hU} \cap \overline V = \emptyset\) and \(\overline{hU} \cup \overline V \ne \widehat X\).
    Moreover, by \cite[Proposition 4.4]{Hamann} there is some $N \in \NN$ such that $g^n(\widehat X - V)$ is contained in $U$ for all \(n \geq N\) (see the definition of \emph{contractive \(G\)-completion} in \cite[\S 2]{Hamann}). 
    Therefore, \(hg^n(X - V)\) is contained in \(hU\).
    Now applying \cite[Lemma~2.6, Corollary~4.5]{Hamann}, we have that $h g^n $ is loxodromic for all \(n \geq N\). 
\end{proof}

\begin{lemma}
\label{lem:quasifixing_ray}
    For any \(K \geq 0\) there is a constant \(M = M(\lambda,c,K) \geq 0\) such that the following is true.
    
    Let $x \in X$ and let $g \in G$ be such that \(\dist(x,gx) \leq K\).
    Suppose that $\Fix(g)$ is non-empty and let $p \in \Fix(g)$.
    If $\ell$ is a $(\lambda, c)$-quasi-geodesic ray in \(X\) emanating from $x$ and converging to $p$ then for all $y \in \ell$ we have that \(\dist(y, gy) \leq M\).
\end{lemma}

\begin{proof}
    Using Lemma~\ref{lem:slim-ideal-triangles}, fix $\delta' \geq 0$ such that any ideal $(\lambda, c)$-quasi-geodesic triangle is $\delta'$-slim. Let $\mu = \mu(\lambda, c) \geq 0$ be the Morse constant for $(\lambda, c)$-quasi-geodesics as in Lemma~\ref{lem:morse}. 
    Consider the $(\lambda,c)$-quasi-geodesic triangle $T = \ell \cup g\ell \cup [x, gx]$, where $[x,gx]$ is some choice of geodesic path connecting $x$ to $gx$. 
    We have that $T$ is $\delta'$-slim. 
    Fix $y \in \ell$ and write $L = \dist(y, x)$. 
    If \(L \leq \delta' + K\), then choosing \(M = 2\delta' + 3K\) would give that 
    \[
    \dist(y, gy) \leq \dist(y,x) + \dist(x,gx) + \dist(gx,gy) \leq 2L + K = M
    \] 
    and we are done. 
    Therefore we suppose otherwise, that \(L > \delta' + K\).
    By the \(\delta'\)-slimness of \(T\) there exists some $z \in \ell \cup [x,gx]$ such that $\dist(gy, z) \leq \delta'$.
    By our assumption on \(L\) and the fact that \(\dist(x,gx) \leq K\), it must be \(z \in \ell\).
    We will show that $\dist(y,z)$ is uniformly bounded above, which implies the lemma.

    Parameterise $\ell$ as a $(\lambda, c)$-quasi-isometric embedding $\ell : [0,\infty) \to X$. Choose $a,b \geq 0$ such that $\ell(a) = y$, $\ell(b) = z$. We assume without loss of generality that $a \leq b$. The other case is similar. 
    Let $q = [x,z]$ be a geodesic connecting $x$ to $z$. Then $y$ is contained in the $\mu$-neighbourhood of $q$. Choose $y' \in q$ such that $\dist(y,y') \leq \mu$. We have that
    \begin{align*}
        \dist(z,x) &= \dist(x,y') + \dist(y',z) \\
        &\geq \dist(x,y) - \mu + \dist(y,z) - \mu \\
        &= \dist(y,z) + L - 2\mu.
    \end{align*}
    We also have that
    \begin{align*}
        \dist(z,x) &\leq \dist(z,gy) + \dist(gy,gx) + \dist(gx,x) \\
            &\leq \delta' + L + K.
    \end{align*}
    Combining the above, we deduce that 
    $$
    \dist(y,z) \leq K + \delta' + 2\mu, 
    $$
    so that lemma follows from setting $M = \max\{2\delta' + 3K, K + 2\delta' + 2\mu\} $. 
\end{proof}

For the remainder of the section we will assume that the action of \(G\) is such that \(\Lambda G = \partial X\).

\begin{lemma}
\label{lem:fix_set_nowhere_dense}
    Given $g \in G$, either $\Fix(g) = \partial X$ or $\Fix(g)$ is nowhere dense. 
\end{lemma}

\begin{proof}
    Suppose that \(\Fix(g) \ne \partial X\) and that it is not nowhere dense.
    Since \(\Fix(g)\) is closed, there is an open subset \(U \subset \Fix(g)\).
    Let \(V \subset \partial X\) be an open set disjoint from \(\Fix(g)\).
    By Theorem~\ref{thm:lox_fixpoints_dense}, there is loxodromic a \(z\) with \(z^\infty \in U\) and \(z^{-\infty} \in V\).
    Now \(z^g\) is a loxodromic with \((z^g)^\infty = g \cdot z^\infty = z^\infty\) but \((z^g)^{-\infty} = g \cdot z^{-\infty} \ne z^{-\infty}\), contradicting Lemma~\ref{lem:lox_fixpoints_disjoint}.
\end{proof}

\begin{lemma}\label{lem:only-fin-intersects}
    Fix $x \in X$ and let $B \subset G$ be a subset with $\diam(Bx) < \infty$. 
    Then there is \(N \in \NN\) such that for every $p \in \partial X$, there exists a neighbourhood $U$ of $p$ such that
    $$
    |\{b \in B : \Fix(b) \cap U \neq \emptyset\}| \leq N.
    $$
    In particular, if $\Fix(b)$ is discrete for every $b \in B$, then so is $\bigcup_{b \in B} \Fix(b)$. 
\end{lemma}

\begin{proof}
    Write \(K = \diam(Bx) < \infty\). Let \(p \in \partial X\). 
    Take \(\ell\) to be a \((1,20\delta)\)-quasigeodesic ray emanating from \(x\) and converging to \(p\).
    By Lemma~\ref{lem:morse}, there is \(\delta' \geq 0\) for which \((1,20\delta)\)-quasigeodesic triangles in \(X \cup \partial X\) are \(\delta'\)-slim.
    By Lemma~\ref{lem:quasifixing_ray}, there is \(M = M(1,20\delta,K) \geq 0\) such that \(\dist(y,by) \leq M\) for any \(y \in \ell\) and \(b \in B\). 
    Let \(R, N > 0\) be the constants provided by acylindricity, applied with \(\varepsilon = M + 2\delta'\).
    Let \(U \subset \partial X\) be a neighbourhood of \(p\) such that for any point \(q\ \in U\), a \((1,20\delta)\)-quasigeodesic ray emanating from \(x\) converging on \(q\) has an initial segment of length \(2R\) that lies in a \(\delta'\)-neighbourhood of an initial segment of \(\ell\) of the same length.
    
    Suppose that there are \(n\) distinct elements \(b_1, \dots, b_n \in B\) with \(q_1. \dots, q_n \in U\), where \(q_i \in \Fix(b_i)\).
    Let \(\ell_i\) be \((1,20\delta)\)-quasigeodesic rays from \(x\) to \(q_i\) for each \(i\), and again observe that by Lemma~\ref{lem:quasifixing_ray}, \(\dist(z,bz) \leq M\) for any \(z \in \ell_i\) and \(b \in B\).
    Take \(y \in \ell\) with \(R < \dist(x,y) < 2R\), so that for each \(i\) there is a point \(z_i \in \ell_i\) with \(\dist(y,z_i) < \delta' \).
    Then for each \(i\), we have
    \[
        \dist(y,b_i y) \leq \dist(y,z_i) + \dist(z_i,b_i z_i) + \dist(b_i z_i, b_i y) < M + 2\delta'. 
    \]
    Of course, \(\dist(x,b_i x) \leq K < M\), so acylindricity implies that \(n \leq N\). The lemma follows. 
\end{proof}

\begin{lemma}
\label{lem:union_of_fix_sets_not_dense}
    Fix $x \in X$ and let $B \subset G$ be a subset with $\diam(Bx) < \infty$. If \(\Fix(b) \ne \partial X\) for all \(b \in B\), then
    $
    \bigcup_{b \in B} \Fix(b)
    $
    is nowhere dense. 
\end{lemma}

\begin{proof}
    Write \(C = \bigcup_{b \in B} \Fix(b)\).
    If the closure \(\overline{C}\) contains an open set \(V\), there is \(p \in V \cap C\).
    Let \(U\) be the neighbourhood of \(p\) provided by Lemma~\ref{lem:only-fin-intersects}, so that \(U\) only meets finitely many of the sets \(\Fix(b)\).
    By Lemma~\ref{lem:fix_set_nowhere_dense}, each \(\Fix(b)\) is nowhere dense.
    Therefore \(U \cap C\) is the union of finitely many nowhere dense subsets, and hence is nowhere dense.
    But then \(U \cap V\) is an open set contained in the closure of \(U \cap C\), a contradiction.
\end{proof}

The following corollary is immediate from Lemma~\ref{lem:only-fin-intersects}. 
We do not use it here, but include it as it may be useful for applications.

\begin{corollary}
    There is \(N \geq 0\) such that if \(g \in G\) is elliptic with \(\Fix(g) \ne \emptyset\), then \(g\) has order at most \(N\). 
\end{corollary}


\section{Swingers and tiles}
\label{sec:swingers}

    For this section, let $G$ be a group, and let $S$ denote a (possibly infinite) generating set. 
    We will denote by \(\abs{\,\cdot\,} = \abs{\,\cdot\,}_S\) the word length of an element of \(G\) with respect to \(S\), and  \(\dist = \dist_S\) for the associated edge-path metric on \(\Gamma(G,S)\).
    Let us introduce the following technical terminology. 

    \begin{definition}[Swingers]
        Given $r> 0$,  we say that an element $z \in G$ is an \textit{$r$-swinger with respect to \(S\)} if $z$ is loxodromic in $\Gamma(G,S)$, and for every \(b \in G\) with \(1 \leq \abs{b} \leq r\) we have
        \[
            \abs{z^{mi} b z^{mj}} > \frac{3}{2}\abs{z^m}
        \]
        for each \(i, j \in \{1,-1\}\), $m \geq 1$.
    \end{definition}

    Note that the choice of \(\tfrac{3}{2}\) in the above is somewhat arbitrary: for the purposes of this article any constant strictly greater than \(1\) would be suitable, and any constant strictly less than \(2\) could be realised.
    
    The nomenclature above is motivated by the dynamics of such elements. 
    One can imagine that conjugating a swinger element $z$ by such short elements $b$ has the effect of `swinging' $z$ around, causing it to point in a totally different direction from both $z$ and $z^{-1}$. 
    The following lemma illustrates this intuition. 

    \begin{lemma}\label{lem:two-ended}
        Let $z \in G$ be an $r$-swinger with respect to $S$. Then for all $b \in G$ with \(1 \leq \abs{b} \leq r\), the subgroup $\langle z, z^b\rangle$ is not virtually cyclic. 
    \end{lemma}

    \begin{proof}
        Suppose otherwise, that $H = \langle z, z^b\rangle$ is virtually cyclic.  
        Then both $\langle z\rangle$ and $\langle z^b \rangle$ have finite index in $H$. Thus the intersection $\langle z\rangle \cap \langle z^b \rangle$ has finite index in $H$, and in particular is non-trivial. 
        Thus there are $n_1, n_2 \ne 0$ such that $(z^{n_1})^b= z^{n_2}$. 
        We treat the case that $n_1 \geq n_2 \geq 1$, for the other cases are similar. 
        Since \(z\) is an \(r\)-swinger, we have that for any \(m \geq 1\),
        \[
            \abs{z^{m(n_1-n_2)}} =\abs{z^{-n_2m}z^{n_1m}} = \abs{(z^{-n_1m})^b z^{n_1m}} \geq |z^{n_1m} b^{-1} z^{n_1m}| - r > \abs{z^{n_1m}} - r. 
        \]
        As \(z\) is loxodromic, the map \(n \mapsto z^n\) is a \((\lambda,c)\)-quasiisometry for some \(\lambda \geq 1\) and \(c \geq 0\).
        The above inequality thus yields
        \[
            \lambda m(n_1 - n_2) + c > \frac{1}{\lambda} mn_1 - c - r
        \]
        for all \(m \geq 1\).
        However, since \(n_1 - n_2 < n_1\), this inequality can only hold for finitely many \(m\): a contradiction.
    \end{proof}

    \begin{definition}
        Let $G$ be a group and let \(S \subset G\) be a generating set.
        We say that \(\Gamma(G,S)\) \emph{admits swingers} if \(G\) contains an \(r\)-swinger with respect to \(S\) for each \(r > 0\).
    \end{definition}

    We will show that if a group \(G\) acts acylindrically on a hyperbolic Cayley graph that admits swingers, then \(G\) is monotileable. 
    The idea behind this argument is essentially due to Akhmedov and appears in \cite{akhmedov2023big}. 
    Recall that a subset \(Y\) of a metric space is \emph{\(r\)-separated} if \(\dist(x,y) \geq r\) for any distinct \(x, y \in Y\).
    We begin with the following lemma.

    \begin{lemma}\label{lem:s-sep-tfi}
        Let \(G\) be a group with generating set \(S\) and suppose that $\Gamma(G,S)$ is $\delta$-hyperbolic. 
        For each $r > 0$, there is \(R = R(r,\delta) > 0\) such that if $z \in G$ is an $r$-swinger with \(\abs{z} \geq R\), the subset 
        \[
            D_{z ,r} = \{x \in G : |xz| < |x| + r \ \ \mathrm{ or } \ \ |xz^{-1}| < |x| + r\} 
        \]
        is $r$-separated and the set
        \[
            E_{z ,r} = \{x \in G : |xz| < |x| + r \ \ \mathrm{ and } \ \ |xz^{-1}| < |x| + r\} 
        \]
        is empty. 
    \end{lemma}

    \begin{proof}
        We first prove the former claim.
        Let \(x, y \in D_{z ,r}\) be distinct elements so that there are \(i, j \in \{1, -1\}\) such that \(\abs{x z^i} < \abs{x} + r\) and \(\abs{y z^j} < \abs{y} + r\).
        Pick geodesics \(\alpha = [1,xz^i], \beta = [1,yz^j]\), and \(\gamma = [xz^i,yz^j]\).
        Suppose for a contradiction that \(\dist(x,y) \leq r\), so that we have \(1 \leq \abs{y^{-1}x} \leq r\).
        Then since \(z\) is an \(r\)-swinger,
        \[
            \abs{\gamma} = \dist(xz^i,yz^j) = \abs{z^{-j}y^{-1}xz^i} > \frac{3}{2}\abs{z}
        \]
        Therefore there is a point \(p\) of \(\gamma\) with 
        \begin{equation}
        \label{eq:dxzi_yzj_p}
            \dist(xz^i,p) \geq \frac{3}{4} \abs{z} \quad \textnormal{ and } \quad \dist(yz^j,p) \geq \frac{3}{4} \abs{z}
        \end{equation}
        
        By the \(\delta\)-slimness of the triangle \(\gamma \cup [x,xz^i] \cup [x,yz^j]\), there is a point \(q\) of \([x,xz^i] \cup [x,yz^j]\) such that \(\dist(p,q) \leq \delta\).
        If \(q\) lies on \([x,yz^j]\) then
        \begin{align*}
            \dist(q,x) &= \dist(x,yz^j) - \dist(q,yz^j) \\
                &\leq \abs{z} + r - (\dist(p,yz^j) - \dist(p,q)) \\
                &\leq \frac{1}{4} \abs{z} + r + \delta
        \end{align*}
        where the last inequality is an application of (\ref{eq:dxzi_yzj_p}).
        It follows that
        \begin{equation}
        \label{eq:dpx_bound}
            \dist(p,x) \leq \dist(p,q) + \dist(q,x) \leq \frac{1}{4} \abs{z} + r + 2\delta.
        \end{equation}
        A similar argument gives the same upper bound when \(q\) lies on \([x,xz^i]\).
        Now by slimness of the triangle \(\alpha \cup \beta \cup \gamma\), the point \(p\) lies in a \(\delta\)-neighbourhood of \(\alpha \cup \beta\).
        We may assume \(p\) lies in the \(\delta\)-neighbourhood of a point \(t\) of \(\alpha\), with the other case similar.
        Since \(\alpha\) is a geodesic representing \(xz\), we have \(\abs{xz^i} = \dist(1,t) + \dist(t,xz^i)\).
        Moreover, by (\ref{eq:dpx_bound}) and the definition of \(t\), we have that 
        \[
            \dist(x,t) \leq \dist(x,p) + \dist(p,t) \leq \frac{1}{4}\abs{z} + r + 3\delta.
        \]
        Then combining the above with the triangle inequality,
        \begin{equation}
        \label{eq:dxzi_bound}
            \abs{xz^i} \geq \dist(1,x) + \dist(x,xz^i) - 2\dist(x,t) \geq \abs{x} + \frac{1}{2}\abs{z} - 2r - 6\delta.
        \end{equation}
        Since \(x \in D_{z ,r}\), the inequality \(\abs{xz^i} < \abs{x} + r\) holds.
        Therefore (\ref{eq:dxzi_bound}) implies that
        \[
            \abs{z} < 6r+12\delta.
        \]
        Now choosing \(R = 6r + 12\delta\) gives us the desired contradiction.

        For the latter claim: if \(x \in E_{z,r}\) then the same argument as above, with \(x = y, i = 1\), and \(j = -1\), again yields a contradiction.
        Thus \(E_{z,r}\) is empty.
    \end{proof}

    We are ready to prove the main statement of this section. 

    \begin{proposition}
    \label{prop:hyp_swingers=>(P)}
        Let \(G\) be a countable group with generating set \(S\).
        Suppose $\Gamma(G,S)$ is $\delta$-hyperbolic and admits swingers.
        If the action of \(G\) on \(\Gamma(G,S)\) is acylindrical,
        then for any finite subset $F \subset G$, there exists $z \in G$ such that $F \cup \{z\}$ is a tile for $G$. 
        In particular, $G$ is monotileable. 
    \end{proposition}

    \begin{proof}
    By translating \(F\) in advance, we may assume that \(1 \in F\).
    If $|F|=1$ then \(F\) is already a tile of \(G\). 
    Thus, we assume that $|F| > 1$, so that there is $v \in F$ with $v \neq 1$. 
    Write \(M = \max\{|g| : g \in F \cup F^{-1}\}\) and let 
    $
        r = 4M + 1
    $.
    Let \(R = R(r,\delta) > 0\) be the constant provided by Lemma~\ref{lem:s-sep-tfi}.
    
    Now, by assumption there is an \(r\)-swinger \(y \in G\). 
    Note that no nontrivial $b \in G$ with $\abs{b} \leq r$ sends one endpoint of $y$ in $\partial X$ to another.
    Indeed, otherwise $y$ and $y^b$ would either share a single fixed point in \(\partial X\), contradicting Lemma~\ref{lem:lox_fixpoints_disjoint}, or generate a virtually cyclic subgroup, contradicting Lemma~\ref{lem:two-ended}. 
    Let \(N \in \NN\) be the number provided by Lemma~\ref{lem:lox-product} such that $v^{-1}y^n$ has infinite order for all \(n \geq N\).
    We take \(z = y^n\) with \(n \geq N\) large enough to ensure \(\abs{z} \geq R\).
    
    Let
    \[
        C = C_{z ,r} = \{x \in G : |xz^{-1}| < |x| + r\}.
    \] 
    By Lemma~\ref{lem:s-sep-tfi}, the set \(D = D_{z ,r}\) is \(r\)-separated and \(E = E_{z,r}\) is empty.
    As \(C \subset D\), the set \(C\) is also \(r\)-separated.
    We observe the following. 

    \begin{claim}\label{claim:c-props-tfi}
        For all $s \in C$, we have $sv^{-1}z \in C$. 
    \end{claim}

    \begin{proof}
        For the first statement, note that
        $v\neq 1$ and $|v| < r$. Since $s\in C \subset D$ and $D$ is $r$-separated, $sv^{-1}\not\in D$. 
        By definition of \(D\), we have $|sv^{-1}z | \geq |sv^{-1}| + r$.  
        Then
        \[
            |(sv^{-1}z)z ^{-1}| = |sv^{-1}| \leq  |sv^{-1}z | - r <  |sv^{-1}z | + r.
        \]
        Thus it follows from the definition of $C$ that $sv^{-1}z \in C$.. 
    \end{proof}

    Claim~\ref{claim:c-props-tfi} tells us that \(C v^{-1} z \subset C\), and we know \(v^{-1} z\) has infinite order by our choice of $z$.
    Hence by Lemma~\ref{lem:two_element_tile}, there exists a subset $C' \subset C$ such that $C$ decomposes as the disjoint union
    \begin{equation}
    \label{claim:c-decomp-tfi}
        C = \bigsqcup_{s \in C'} \{s, sv^{-1} z\}. 
    \end{equation}
        
    We now begin tiling our group. Our candidate tile will be $F\cup \{z\}$. For ease of notation we write $T = F\cup \{z\}$.
    Let
    $$
    A = \bigcup_{s \in C'} sv^{-1} T. 
    $$
    Observe that for any \(s \in C\), \(\{s,sv^{-1}z\} \subset sv^{-1} T\), since \(\{v,z\} \subset T\).
    Hence (\ref{claim:c-decomp-tfi}) shows \(C \subset A\). 
    We will show that the above union is in fact a disjoint union.
    
    \begin{claim}\label{claim:a-tile}
        For any distinct $s, s' \in C'$ we have $s v^{-1} T \cap s' v^{-1}T = \emptyset$. 
    \end{claim}

    \begin{proof}
        Suppose that $s v^{-1} T \cap s' v^{-1}T$ is non-empty. As \(s \ne s'\), one of the following must hold:
        \begin{enumerate}
            \item $sv^{-1} F \cap s' v^{-1}F \neq \emptyset$, or
            \item $sv^{-1} z \in s' v^{-1}F$, or
            \item $s'v^{-1} z \in s v^{-1}F$. 
        \end{enumerate}
        If the first case holds, then 
        \(
            \dist(s,s') \leq 4 M < r,
        \)
        contradicting the fact that $C$ is $r$-separated. 
        Suppose the second case holds, so that $sv^{-1} z \in s' v^{-1}F$. 
        Similarly to the first case, this implies that \(\dist(s',sv^{-1}z) < r\).
        Observing that $s' \in C, sv^{-1} z \in C$ by Claim~\ref{claim:c-props-tfi}, and that $C$ is $r$-separated, it must be the case that $sv^{-1}z = s'$. 
        But this contradicts the fact that \(\{s,sv^{-1}z\}\) and \(\{s',s'v^{-1}z\}\) are disjoint as in (\ref{claim:c-decomp-tfi}). 
        The final case may be dealt with identically. 
    \end{proof}

    In other words, we have that the translates of $T$ by $\{s v^{-1} : s\in C'\}$ are pairwise disjoint and cover $A$.
    We now proceed to tile $G - A$ by translates of $T$ in a na\"ive way, picking remaining elements of minimal word length and covering them with the images of \(z\) in \(T\). 
    We show that the tiles we obtain in this fashion are disjoint from both \(A\) and from one another.

    \begin{claim}\label{claim:disjoint-new-tiles}
        If $b \notin A$, then \(bz^{-1} T \cap A = \emptyset\).
        Furthermore, if \(c \notin A \sqcup bz^{-1}T\) and \(\abs{b} \leq \abs{c}\), then \(b z^{-1} T \cap c z^{-1} T = \emptyset\).
    \end{claim}

    \begin{proof}
        To begin, note that for any $x \notin A$, we have \(x \notin C\).
        It follows from the definition of \(C\) that
        \begin{equation}
        \label{eq:xiz>xi+r}
            \abs{x z^{-1}} \geq \abs{x} + r.
        \end{equation}
        An application of (\ref{eq:xiz>xi+r}) gives that
        \begin{equation}
        \label{eq:xz-1_in_D}
            \abs{x z^{-1} z} = \abs{x} \leq \abs{xz^{-1}} - r < \abs{xz^{-1}} + r.
        \end{equation}
        Thus, $x z^{-1} \in D$ for any \(x \notin A\).
        In particular, \(bz^{-1} \in D\) and \(cz^{-1} \in D\) for \(b\) and \(c\) as in the statement of the claim.

        Let \(b \notin A\) be an arbitrary element and suppose that $b z^{-1} T \cap A$ is non-empty. 
        By the definition of $A$ we have that $b z^{-1} T \cap s v^{-1} T \neq \emptyset$ for some $s \in C'$. By construction, we have that $b \notin A \supset sv^{-1} T$, and so either
        \begin{enumerate}
            \item $b z^{-1} F \cap sv^{-1} F \neq \emptyset$, or

            \item $sv^{-1} z \in b z^{-1} F $. 
        \end{enumerate}
        In both cases we will deduce a contradiction. 
        Note that by Claim~\ref{claim:c-props-tfi} we have that $sv^{-1} z \in C \subset D$. Recall that $b z^{-1} \in D$ since \(b \notin A\).
        If \(bz^{-1} = s \in C'\), then \(bz^{-1} \in C\), whence by (\ref{eq:xz-1_in_D}), we in fact have \(bz^{-1} \in E\).
        However, this set is empty by Lemma~\ref{lem:s-sep-tfi}, so we must have \(s \neq bz^{-1}\).
        In the first case above, note that $\dist(sv^{-1}, bz^{-1}) \leq 2M$, and so we may deduce that $\dist(s, b z^{-1}) \leq 3M < r$. 
        But this contradicts the fact that $D$ is $r$-separated. 
        In the second case, we see that $\dist(s v^{-1} z, b z^{-1}) < r$. Similarly, this contradicts either the fact that $D$ is $r$-separated or the fact that \(E\) is empty. 
        Therefore \(bz^{-1} T\) and \(A\) are disjoint.
        
        Now suppose there is an element $c \notin A \sqcup b z^{-1} T$ such that $b z^{-1} T \cap c z^{-1} T$ is not empty. Since \(b \neq c\) one of the following three cases must hold: 
        \begin{enumerate}
            \item $b \in c z^{-1}F$, or
    
            \item $c \in b z^{-1}F$, or
    
            \item $b z^{-1} F \cap c z^{-1} F \neq \emptyset$.
        \end{enumerate}
        For the first case we have that \(b = c z^{-1} t\) for some \(t \in F\). Then
        $$
            \abs{b} = \abs{c z^{-1} t} \geq \abs{c z^{-1}} - \abs{t} > \abs{c} + r - r = \abs{c} \geq \abs{b},
        $$
        where the middle inequality follows from (\ref{eq:xiz>xi+r}) and the fact that $\abs{t} < r$ since \(t \in F\). 
        The second case immediately contradicts the construction of $c$, since $c \notin b z^{-1} T \supset b z^{-1}F$.
        In the third case, recall that \(bz^{-1} \in D\) and $c z^{-1} \in D$. 
        However, by assumption we have that 
        \(
            \dist(b z^{-1}, c z^{-1}) \leq 2M < r,
        \)
        which contradicts the fact that $D$ is $r$-separated.  
    \end{proof}

    Write \(A_0 = \{1\}\) and \(A_1 = A\).
    We will proceed with tiling \(G - A\) by induction.
    Let \(i \geq 1\) and suppose we have constructed a set \(A_i\). 
    Let \(n \in \NN\) be the minimal natural number for which there is \(b \notin A_i\) with \(\abs{b} = n\).
    We claim there is a set \(A_{i+1}\), formed of a disjoint union of \(A_i\) and pairwise disjoint translates of \(T\), satisfying
    $$
    \min\{\abs{g} : g \notin A_{i+1}\} > n.
    $$
    Indeed, since \(G\) is countable we may enumerate the elements of \(G - A_i\) with word length \(n\).
    Let \(\{g_1, g_2, \dots\}\) be such an enumeration and write \(A_i^{(1)}= A_i\) and \(n_1 = 1\).
    Now for any \(j > 1\), let \(n_j\) be the smallest natural number such that \(g_{n_j} \notin \bigcup_{k < j} g_{n_k} z^{-1} T\).
    Claim~\ref{claim:disjoint-new-tiles} tells us that \(A_i^{(j-1)}\) and \(g_{n_j} z^{-1} T\) are disjoint, and also by induction that \(g_{n_l} z^{-1} T\) and \(g_{n_k} z^{-1} T\) are pairwise disjoint for all \(1 \leq k < l \leq j\).
    Now let \[A_i^{(j)} = A_i^{(j-1)} \sqcup g_{n_j} z^{-1} T.\]
    Setting \(A_{i+1} = \bigcup_{j \in \NN} A_i^{(j)}\), we readily see that \(g_k \in A_{i+1}\) for all \(k \in \NN\).
    Hence \(A_{i+1}\) satisfies the desired criterion.
    Now by construction, the sets \(\{A_i : i \in \NN\}\) form a nested family of subsets that exhaust \(G\). 
    Since each \(A_i\) is comprised of disjoint translates of \(T = F \cup \{z\}\), the set \(F \cup \{z\}\) is a tile for \(G\).
\end{proof}

\begin{remark}
    We note that in the proof of Proposition~\ref{prop:hyp_swingers=>(P)}, we only use the boundedness of the set \(F\) with respect to \(\dist_S\), rather than the finiteness.
    Combined with the results of the next section, one sees that if \(G\) is acylindrically hyperbolic with no nontrivial finite normal subgroup, then for any bounded subset \(F \subset G\), there is \(z \in G\) such that \(F \cup \{z\}\) is a tile for \(G\).
    It may be of interest, for example then, that any subset of a hyperbolically embedded subgroup of a torsion-free group \(G\) can be extended to a tile of \(G\) by adding a single element.
\end{remark}


\section{Finding swingers in acylindrically hyperbolic groups}
\label{sec:acyl_hyp}

In this section we take \(G\) to be an acylindrically hyperbolic group.
As such, there is a generating set \(S\) of \(G\) such that the Cayley graph \(X = \Gamma(G,S)\) is \(\delta\)-hyperbolic and the action of \(G\) on \(X\) by left translation is acylindrical, nonelementary, and cobounded. 
Again, we will denote by \(\abs{\,\cdot\,} = \abs{\,\cdot\,}_S\) the word length of an element of \(G\) with respect to \(S\), and  \(\dist = \dist_S\) for the associated edge-path metric on \(\Gamma(G,S)\).

\begin{lemma}\label{lem:disjoint-limits}
    Suppose that \(B \subset G\) is a bounded subset of \(G\) such that no element of \(B\) fixes \(\partial X\) pointwise.
    Then there is a loxodromic element \(z \in G\) and an open set $U$ containing $\Fix(z)$ such that \(U \cap \Fix(z^b) = \emptyset\) for all \(b \in B\).
\end{lemma}

\begin{proof}
    Write \(B' = B \cup B^2\), and note that \(\diam(B') < \infty\).
    By Lemma~\ref{lem:union_of_fix_sets_not_dense}, the set \(C = \bigcup_{b \in B'} \Fix(b)\) is not dense in \(\partial X\).
    Hence there is an open set \(V \subset \partial X\) disjoint from \(C\).
    Applying Theorem~\ref{thm:lox_fixpoints_dense}, there is a loxodromic element \(z \in G\) with \(\Fix(z) \subset V\).
    We will first show that no element of \(B\) setwise fixes \(\Fix(z)\).
    
    Let \(b \in B\) and suppose that \(z^\infty = (z^b)^\infty\).
    Now \((z^b)^\infty\) is exactly the translate \(b \cdot z^\infty\), and so \(b\) fixes \(z^\infty\).
    Thus,  \(z^\infty \in C\), which contradicts the fact that \(z^\infty \in V\).
    Similarly, if \(z^{-\infty} = (z^b)^{-\infty}\) then \(z^{-\infty} \in C\), which is again a contradiction.

    It remains to show that \((z^b)^\infty \ne z^{-\infty}\) or \((z^b)^{-\infty} \ne z^{\infty}\).
    If either of these conditions fails, then \(z^b\) and \(z\) share the same fixed points by Lemma~\ref{lem:lox_fixpoints_disjoint}.
    Thus \(b\) permutes the set \(\{z^\infty, z^{-\infty}\}\), and so \(b^2\) fixes it.
    But then \(\{z^\infty, z^{-\infty}\} \subset C\), contradicting the construction of \(z\).

    Since $B$ is bounded, we have that $\diam(\{z^b : b \in B\}) < \infty$. 
    Moreover, $\Fix(z^b)$ is discrete for every $b \in B$, so by Lemma~\ref{lem:only-fin-intersects} we have that 
    $$
    A = \bigcup_{b \in B} \Fix(z^b) 
    $$
    is discrete. In particular, since $\Fix(z)$ is discrete and disjoint from $A$, there is some neighbourhood $U$ of $\Fix(z)$ which is disjoint from $A$. 
\end{proof}

We show that the dynamical condition obtained above can be used to find a swinger. The following lemma is standard.

\begin{lemma}\label{lem:lox_diverging-rays}
    For every $\lambda \geq 1$, $c \geq 0$, and $s \geq 0$, there exists $t = t(\lambda, c, s) \geq 0$ such that the following holds. Let $x, y  \in G$ be loxodromic elements such that 
    $
    \langle x^\infty \cdot y^\infty \rangle_1 < s
    $, and the inclusions of $\langle x\rangle$ and $\langle y \rangle$ into $G$ are $(\lambda, c)$-quasiisometric embeddings. 
    Then 
    \[
        \dist(x^m, y^{n}) > |x^m| + |y^n| - t. 
    \]
    for any \(n, m \geq 1\).
\end{lemma}

\begin{proof}
    Since $\langle x^\infty \cdot y^\infty \rangle_1 < s$, we have by definition that
    \[
        \liminf_{i,j\to \infty} \langle x^i \cdot  y^j \rangle_1 < s.
    \]
    Let $m, n \geq 1$. Choose $M, N \geq 1$ such that $m \leq M$, $n \leq N$, and
    $
    \langle x^M\cdot  y^N \rangle_1 < s
    $. 
    Let $p = [1, x^M]$, $q = [1, y^N]$ be geodesics. By the Morse lemma, we have that there exists $\mu = \mu(\lambda, c) > 0$ such that $x^m$ lies in the $\mu$-neighbourhood of $p$ and $y^n$ lies in the $\mu$-neighbourhood of $q$. 
    Fix $a$ and \(b\) be points on \(p\) and \(q\) such that 
    $\dist(x^m, a) \leq \mu$ and $\dist(y^n, b) \leq \mu$. 
    Since \(p\) and \(q\) are geodesics and \(\langle x^M\cdot  y^N \rangle_1 < s\), we have that $\langle a \cdot b \rangle_1 < s$ also. 
    Thus, by the definition of the Gromov product, we see that 
    $$
    \dist(a,b) > \abs{a} + \abs{b} - 2s. 
    $$
    In particular, we deduce that
    $$
    \dist(x^m,y^n) > \abs{x^m} + \abs{y^n} - 2s - 4\mu. 
    $$
    Setting $t = 2s + 4\mu$, we are done. 
\end{proof}

\begin{lemma}\label{lem:wpd_b-lines}
    Let $B$ be a bounded subset of $G - \{1\}$ and let $z \in G$ be a loxodromic element.
    Suppose that $U$ is an open subset containing $\Fix(z)$ such that for every $b \in B$ we have that \(U \cap \Fix(z^b) = \emptyset\). 
    Then there is \(M \in \NN\) such that for any $m \geq M$, $i, j \in \{\pm 1\}$, \(b \in B\) we have that 
    $$
        \abs{z^{im} b z^{jm}} > \frac{3}{2} \abs{z^m}.
    $$
\end{lemma}

\begin{proof}
    We consider the case where $i = j = 1$, and note that the others follow from identical reasoning. 
    
    Let \(s > 0\) be the supremum of all \(s\) such that the basic open sets \(U(z^\infty, s)\) are contained in \(U\) (see \cite[Definition 2.13]{Benakli_Kapovich} for definitions). 
    Then for all $b \in B$, we have that
    $
    \langle z^\infty \cdot (z^b)^{-\infty}\rangle_1 < s. 
    $
    As $B$ is bounded, there exist uniform constants $\lambda \geq 1$, $c \geq 0$ such that for all $b \in B \cup \{1\}$, the inclusions of $\langle z^b \rangle$ into $G$ are $(\lambda, c)$-quasiisometric embeddings. Let $t = t(\lambda, c, s)$ be as in Lemma~\ref{lem:lox_diverging-rays}, then for every $m \geq 0$, $b \in B$ we have that 
    \begin{equation}
    \label{eq:bzbzbd}
        \dist(bz^{-m}b^{-1}, z^{m}) >\abs{bz^{-m}b^{-1}} + \abs{z^m} - t. 
    \end{equation}
    Let $R = \max\{\abs{b} : b \in B \cup B^{-1}\} < \infty$. Now, fix some arbitrary $b \in B$. We have that $\dist(z^b, bz) = \dist(1,b) \leq R$. Hence, we see that 
    \[\dist(bz^m,z^{-m}) \geq \dist(bz^{m}b^{-1},z^m) - R.\]
    Combining this with (\ref{eq:bzbzbd}), we obtain
    \[
        \abs{z^m b z^m} = \dist(bz^m, z^{-m}) > \abs{bz^m} + \abs{z^m} - t - 2R.
    \]
    We also have by the triangle inequality that $\abs{bz^m} \geq \abs{z^m} - R$. Hence, we conclude that \(\abs{z^m b z^m} = \dist(bz^m, z^{-m}) \geq 2\abs{z^m} - t - 3R\).
    Rearranging, we write
    \begin{equation}
    \label{eq:zm_lower_bound}
        \abs{z^m b z^m} \geq \abs{z^m} + (\abs{z^m} - t - 3R).
    \end{equation}
    As \(z\) is a loxodromic element, \(\abs{z^m} \to \infty\) as \(m \to \infty\).
    That is, there is \(M \in \NN\) such that \(\tfrac{1}{2}\abs{z^m} > t + 3R\) for all \(m \geq M\).
    Hence by (\ref{eq:zm_lower_bound}) we have \(\abs{z^m b z^m} > \tfrac{3}{2}\abs{z^m}\). Since $b \in B$ was arbitrary and $t$, $R$ and $M$ do not depend on the choice of $b$, the result follows. 
\end{proof}

\begin{proposition}
\label{prop:acyl_hyp_swingers}
    Suppose that \(G\) contains no nontrivial finite normal subgroups.
    Then \(X = \Gamma(G,S)\) admits swingers.
\end{proposition}

\begin{proof}
    Let \(r > 0\) and \(B = \{b \in G : 1 \leq \abs{b} \leq r\}\).
    We show that \(G\) contains an \(r\)-swinger \(z \in G\).
    Together with the assumption that \(G\) contains no finite normal subgroup, Corollary~\ref{cor:acyl_boundary_stab_finite} tells us that no nontrivial elements of \(G\) pointwise fix \(\partial X\).
    Thus we may apply Lemma~\ref{lem:disjoint-limits} to find a loxodromic element \(y \in G\) and an open subset \(U\) containing \(\Fix(y)\) such that \(U \cap \Fix(y^b) = \emptyset\) for all \(b \in B\).
    Now Lemma~\ref{lem:wpd_b-lines} tells us that there is \(M \in \NN\) such that \(\abs{y^{im} b y^{jm}} > \tfrac{3}{2}\abs{y^m}\) for all \(m \geq M\), \(i,j \in \{1,-1\}\), and \(b \in B\).
    It follows that \(z = y^M\) is an \(r\)-swinger with respect to \(S\), completing the proof.
\end{proof}


\section{Main results}
\label{sec:conclusion}

We may now prove the main results from the introduction.

\begin{proof}[Proof of Theorem~\ref{thmA}]
    We begin by showing that a countable acylindrically hyperbolic group \(G\) is monotileable.
    By \cite[Theorem 6.14]{DGO}, \(G\) contains a maximal finite normal subgroup \(K \lhd G\).
    The quotient \(G/K\) is again acylindrically hyperbolic \cite[Lemma 3.9]{Minasyan_Osin}, and contains no nontrivial finite normal subgroups.
    Now combining Propositions~\ref{prop:hyp_swingers=>(P)} and \ref{prop:acyl_hyp_swingers} shows that \(G/K\) is monotileable.
    As monotileability is stable under extensions by finite groups, \(G\) is monotileable.

    Now let \(G\) be an arbitrary acylindrically hyperbolic group.
    Let \(F \subset G\) be a finite subset.
    Since \(G\) is acylindrically hyperbolic, it contains at least two independent loxodromic elements \(g, h \in G\).
    Then by Theorem~\ref{thm:osin_trichotomy}, \(H = \langle F, g, h \rangle\) is a countable acylindrically hyperbolic subgroup of \(G\) containing \(F\).
    Therefore \(F\) extends to a finite tile of \(H\).
    As the cosets of \(H\) are disjoint and cover \(G\), it follows that \(F\) is contained in a finite tile of \(G\).
\end{proof}

\begin{proof}[Proof of Corollary~\ref{cor:one_relator}]
    Let \(G\) be a group with one-relator presentation $\langle S \ | \ r\rangle$.
    If \(S\) contains three or more elements then \(G\) is acylindrically hyperbolic \cite[Corollary 2.6]{Minasyan_Osin}, whence it is monotileable by Theorem~\ref{thmA}.
    Otherwise, combining \cite[Proposition 4.21]{Minasyan_Osin} and \cite[Theorem 3.2]{Button_Kropholler}, \(G\) is either acylindrically hyperbolic, a generalised Baumslag--Solitar group, or a mapping torus of an injective endomorphism of a finitely generated free group.

    In the first case, we are done by Theorem~\ref{thmA} again.
    In the second case, we observe that generalised Baumslag--Solitar groups are free-by-metabelian by a result of P. Kropholler \cite[Theorem C, Corollary 2]{Kropholler_GBS}.
    Both free groups and metabelian groups are monotileable, so free-by-metabelian groups are too.
    In the final case, such groups are known to be residually finite by a result of Borisov and Sapir \cite[Theorem 1.2]{Borisov_Sapir}.
    As residually finite groups are monotileable, the last case is covered also.
\end{proof}

\begin{proof}[Proof of Corollary~\ref{cor:artin}]
    Let \(A_\Gamma\) be a two-dimensional Artin group.
    Recall that an Artin group is called \emph{reducible} if \(\Gamma\) is a join of subgraphs \(\Gamma_1\) and \(\Gamma_2\), and any edge \(e \in E\Gamma - (E\Gamma_1 \cup E\Gamma_2)\) has label 2, and \emph{irreducible} otherwise.
    If \(A_\Gamma\) is reducible, then \(A_\Gamma \cong A_{\Gamma_1} \times A_{\Gamma_2}\).
    Being monotileable is closed under products, so we may restrict our attention to irreducible Artin groups.
    
    If \(\Gamma\) has a single vertex, then \(A_\Gamma\) is a cyclic group.
    If \(\Gamma\) has two vertices, then it is called a \emph{dihedral Artin group}, and it is known that \(A_\Gamma\) is an extension of a cyclic group by a free product of cyclic groups \cite[Section 2]{CHR}.
    Since being monotileable is closed under extensions and free products, \(A_\Gamma\) is monotileable.
    Finally, if \(\Gamma\) has more than three vertices, then \(A_\Gamma\) is acylindrically hyperbolic \cite[Theorem A]{Vaskou}, whence Theorem~\ref{thmA} applies.
\end{proof}

We conclude with some questions.
In \cite[Theorem 1.5]{Sisto}, Sisto shows that non-virtually cyclic \(\operatorname{CAT}(0)\) groups containing a rank one isometry are acylindrically hyperbolic.
Hence Theorem~\ref{thmA} applies to many \(\operatorname{CAT}(0)\) groups.
However, there are \(\operatorname{CAT}(0)\) groups that are not known to be monotileable (for example, the groups constructed in \cite{Burger_Mozes} are finitely presented, \(\operatorname{CAT}(0)\), simple, and not acylindrically hyperbolic).

\begin{question}
    Are \(\operatorname{CAT}(0)\) groups monotileable?
\end{question}

The rank-rigidity conjecture posits that an infinite discrete group \(G\) acting geometrically on an irreducible \(\operatorname{CAT}(0)\) space \(X\) either contains a rank one isometry, or else \(X\) is a higher rank Riemannian symmetric space or Euclidean building \cite{Ballman_Buyalo}.
It may be possible to leverage the geometry of the spaces in the latter category to understand tilings.

\begin{question}
    If \(G\) is a group acting geometrically on a (higher rank) symmetric space or (Euclidean) building, is \(G\) monotileable?
\end{question}

Automatic and biautomatic groups form large classes of groups with auspicious algorithmic and geometric properties.
Hyperbolic groups are known to be (bi)automatic, so it is natural to consider generalising to these classes.

\begin{question}
    Are automatic groups monotileable? Or biautomatic groups?
\end{question}

Lastly, recall that many groups have been shown to have the stronger, uniform tiling property known as CCC \cite{Gao_Jackson_Seward}.
The geometric techniques in this paper allow for tilings that are very far from being uniform, but it would be interesting to see if they can be modified to obtain the CCC property.

\begin{question}
    Do (acylindrically) hyperbolic groups have CCC tilings?
\end{question}

\bibliographystyle{abbrv}
\bibliography{bibliography}

@article{akhmedov2023big,
  title={Big Tiles in Hyperbolic Groups},
  author={Akhmedov, Azer},
  journal={arXiv preprint arXiv:2309.02607},
  year={2023}
}

@article {Bowditch_Tight_Geods,
    AUTHOR = {Bowditch, Brian H.},
     TITLE = {Tight geodesics in the curve complex},
   JOURNAL = {Invent. Math.},
  FJOURNAL = {Inventiones Mathematicae},
    VOLUME = {171},
      YEAR = {2008},
    NUMBER = {2},
     PAGES = {281--300},
      ISSN = {0020-9910,1432-1297},
   MRCLASS = {57M50 (20F65)},
  MRNUMBER = {2367021},
MRREVIEWER = {Jason\ A.\ Behrstock},
       DOI = {10.1007/s00222-007-0081-y},
       URL = {https://doi.org/10.1007/s00222-007-0081-y},
}

@article {Osin_Acyl_hyp,
    AUTHOR = {Osin, D.},
     TITLE = {Acylindrically hyperbolic groups},
   JOURNAL = {Trans. Amer. Math. Soc.},
  FJOURNAL = {Transactions of the American Mathematical Society},
    VOLUME = {368},
      YEAR = {2016},
    NUMBER = {2},
     PAGES = {851--888},
      ISSN = {0002-9947,1088-6850},
   MRCLASS = {20F67 (20F65)},
  MRNUMBER = {3430352},
MRREVIEWER = {Alessandro\ Sisto},
       DOI = {10.1090/tran/6343},
       URL = {https://doi.org/10.1090/tran/6343},
}

@article {DGO,
    AUTHOR = {Dahmani, F. and Guirardel, V. and Osin, D.},
     TITLE = {Hyperbolically embedded subgroups and rotating families in
              groups acting on hyperbolic spaces},
   JOURNAL = {Mem. Amer. Math. Soc.},
  FJOURNAL = {Memoirs of the American Mathematical Society},
    VOLUME = {245},
      YEAR = {2017},
    NUMBER = {1156},
     PAGES = {v+152},
      ISSN = {0065-9266,1947-6221},
      ISBN = {978-1-4704-2194-6; 978-1-4704-3601-8},
   MRCLASS = {20F65 (20F06 20F34 20F67 57M07)},
  MRNUMBER = {3589159},
MRREVIEWER = {Dominik\ Gruber},
       DOI = {10.1090/memo/1156},
       URL = {https://doi.org/10.1090/memo/1156},
}

@article {Hamann,
    AUTHOR = {Hamann, Matthias},
     TITLE = {Group actions on metric spaces: fixed points and free
              subgroups},
   JOURNAL = {Abh. Math. Semin. Univ. Hambg.},
  FJOURNAL = {Abhandlungen aus dem Mathematischen Seminar der
              Universit\"{a}t Hamburg},
    VOLUME = {87},
      YEAR = {2017},
    NUMBER = {2},
     PAGES = {245--263},
      ISSN = {0025-5858,1865-8784},
   MRCLASS = {20F67 (05C63)},
  MRNUMBER = {3696149},
MRREVIEWER = {Cheryl\ E.\ Praeger},
       DOI = {10.1007/s12188-016-0164-z},
       URL = {https://doi.org/10.1007/s12188-016-0164-z},
}

@incollection {Benakli_Kapovich,
    AUTHOR = {Kapovich, Ilya and Benakli, Nadia},
     TITLE = {Boundaries of hyperbolic groups},
 BOOKTITLE = {Combinatorial and geometric group theory ({N}ew {Y}ork,
              2000/{H}oboken, {NJ}, 2001)},
    SERIES = {Contemp. Math.},
    VOLUME = {296},
     PAGES = {39--93},
 PUBLISHER = {Amer. Math. Soc., Providence, RI},
      YEAR = {2002},
      ISBN = {0-8218-2822-3},
   MRCLASS = {20F67},
  MRNUMBER = {1921706},
       DOI = {10.1090/conm/296/05068},
       URL = {https://doi.org/10.1090/conm/296/05068},
}

@book{dructu2018geometric,
  title={Geometric group theory},
  author={Dru{\c{t}}u, Cornelia and Kapovich, Michael},
  volume={63},
  year={2018},
  publisher={American Mathematical Soc.}
}

@article {Chou,
    AUTHOR = {Chou, Ching},
     TITLE = {Elementary amenable groups},
   JOURNAL = {Illinois J. Math.},
  FJOURNAL = {Illinois Journal of Mathematics},
    VOLUME = {24},
      YEAR = {1980},
    NUMBER = {3},
     PAGES = {396--407},
      ISSN = {0019-2082},
   MRCLASS = {43A07 (20E99)},
  MRNUMBER = {573475},
MRREVIEWER = {J.\ Wiegold},
       URL = {http://projecteuclid.org/euclid.ijm/1256047608},
}

@incollection {Weiss,
    AUTHOR = {Weiss, Benjamin},
     TITLE = {Monotileable amenable groups},
 BOOKTITLE = {Topology, ergodic theory, real algebraic geometry},
    SERIES = {Amer. Math. Soc. Transl. Ser. 2},
    VOLUME = {202},
     PAGES = {257--262},
 PUBLISHER = {Amer. Math. Soc., Providence, RI},
      YEAR = {2001},
      ISBN = {0-8218-2740-5},
   MRCLASS = {22D40},
  MRNUMBER = {1819193},
       DOI = {10.1090/trans2/202/18},
       URL = {https://doi.org/10.1090/trans2/202/18},
}

@article {Button_Kropholler,
    AUTHOR = {Button, J. O. and Kropholler, R. P.},
     TITLE = {Nonhyperbolic free-by-cyclic and one-relator groups},
   JOURNAL = {New York J. Math.},
  FJOURNAL = {New York Journal of Mathematics},
    VOLUME = {22},
      YEAR = {2016},
     PAGES = {755--774},
      ISSN = {1076-9803},
   MRCLASS = {20F67 (20E22 20F05 57M20)},
  MRNUMBER = {3548122},
MRREVIEWER = {R\'{e}mi\ Bernard\ Coulon},
       URL = {http://nyjm.albany.edu:8000/j/2016/22_755.html},
}

@article {Borisov_Sapir,
    AUTHOR = {Borisov, Alexander and Sapir, Mark},
     TITLE = {Polynomial maps over finite fields and residual finiteness of
              mapping tori of group endomorphisms},
   JOURNAL = {Invent. Math.},
  FJOURNAL = {Inventiones Mathematicae},
    VOLUME = {160},
      YEAR = {2005},
    NUMBER = {2},
     PAGES = {341--356},
      ISSN = {0020-9910,1432-1297},
   MRCLASS = {20E06 (20E26)},
  MRNUMBER = {2138070},
MRREVIEWER = {Martin\ Edjvet},
       DOI = {10.1007/s00222-004-0411-2},
       URL = {https://doi.org/10.1007/s00222-004-0411-2},
}

@article {Kropholler_GBS,
    AUTHOR = {Kropholler, P. H.},
     TITLE = {Baumslag-{S}olitar groups and some other groups of
              cohomological dimension two},
   JOURNAL = {Comment. Math. Helv.},
  FJOURNAL = {Commentarii Mathematici Helvetici},
    VOLUME = {65},
      YEAR = {1990},
    NUMBER = {4},
     PAGES = {547--558},
      ISSN = {0010-2571,1420-8946},
   MRCLASS = {20F32 (20J05 57M05)},
  MRNUMBER = {1078097},
MRREVIEWER = {Martin\ A.\ Roller},
       DOI = {10.1007/BF02566625},
       URL = {https://doi.org/10.1007/BF02566625},
}

@article {Minasyan_Osin,
    AUTHOR = {Minasyan, Ashot and Osin, Denis},
     TITLE = {Acylindrical hyperbolicity of groups acting on trees},
   JOURNAL = {Math. Ann.},
  FJOURNAL = {Mathematische Annalen},
    VOLUME = {362},
      YEAR = {2015},
    NUMBER = {3-4},
     PAGES = {1055--1105},
      ISSN = {0025-5831,1432-1807},
   MRCLASS = {20F67 (20E06 20E08 20F65 57M05)},
  MRNUMBER = {3368093},
MRREVIEWER = {Mohammad\ Shahryari},
       DOI = {10.1007/s00208-014-1138-z},
       URL = {https://doi.org/10.1007/s00208-014-1138-z},
}

@article {Seward_PT,
    AUTHOR = {Seward, Brandon},
     TITLE = {Burnside's {P}roblem, spanning trees and tilings},
   JOURNAL = {Geom. Topol.},
  FJOURNAL = {Geometry \& Topology},
    VOLUME = {18},
      YEAR = {2014},
    NUMBER = {1},
     PAGES = {179--210},
      ISSN = {1465-3060,1364-0380},
   MRCLASS = {20F65 (05C25 05C63)},
  MRNUMBER = {3158775},
MRREVIEWER = {Sang-hyun\ Kim},
       DOI = {10.2140/gt.2014.18.179},
       URL = {https://doi.org/10.2140/gt.2014.18.179},
}

@article {Gao_Jackson_Seward,
    AUTHOR = {Gao, Su and Jackson, Steve and Seward, Brandon},
     TITLE = {Group colorings and {B}ernoulli subflows},
   JOURNAL = {Mem. Amer. Math. Soc.},
  FJOURNAL = {Memoirs of the American Mathematical Society},
    VOLUME = {241},
      YEAR = {2016},
    NUMBER = {1141},
     PAGES = {vi+241},
      ISSN = {0065-9266,1947-6221},
      ISBN = {978-1-4704-1847-2; 978-1-4704-2875-4},
   MRCLASS = {37B10 (03E15 20F65 37B05)},
  MRNUMBER = {3478758},
MRREVIEWER = {Tam\'{a}s\ M\'{a}trai},
       DOI = {10.1090/memo/1141},
       URL = {https://doi.org/10.1090/memo/1141},
}

@article {Sisto,
    AUTHOR = {Sisto, Alessandro},
     TITLE = {Contracting elements and random walks},
   JOURNAL = {J. Reine Angew. Math.},
  FJOURNAL = {Journal f\"{u}r die Reine und Angewandte Mathematik. [Crelle's
              Journal]},
    VOLUME = {742},
      YEAR = {2018},
     PAGES = {79--114},
      ISSN = {0075-4102,1435-5345},
   MRCLASS = {20F65 (20F67 60G50)},
  MRNUMBER = {3849623},
MRREVIEWER = {Enric\ Ventura Capell},
       DOI = {10.1515/crelle-2015-0093},
       URL = {https://doi.org/10.1515/crelle-2015-0093},
}

@article {Akhemdoc_Fulghesu,
    AUTHOR = {Akhmedov, Azer and Fulghesu, Damiano},
     TITLE = {Arithmetic sets in groups},
   JOURNAL = {Math. Z.},
  FJOURNAL = {Mathematische Zeitschrift},
    VOLUME = {292},
      YEAR = {2019},
    NUMBER = {3-4},
     PAGES = {1195--1206},
      ISSN = {0025-5874,1432-1823},
   MRCLASS = {05B45 (20F99)},
  MRNUMBER = {3980288},
MRREVIEWER = {Sean\ Eberhard},
       DOI = {10.1007/s00209-018-2125-y},
       URL = {https://doi.org/10.1007/s00209-018-2125-y},
}

@incollection {Gromov_hyperbolic,
    AUTHOR = {Gromov, M.},
     TITLE = {Hyperbolic groups},
 BOOKTITLE = {Essays in group theory},
    SERIES = {Math. Sci. Res. Inst. Publ.},
    VOLUME = {8},
     PAGES = {75--263},
 PUBLISHER = {Springer, New York},
      YEAR = {1987},
      ISBN = {0-387-96618-8},
   MRCLASS = {20F32 (20F06 20F10 22E40 53C20 57R75 58F17)},
  MRNUMBER = {919829},
MRREVIEWER = {Christopher\ W.\ Stark},
       DOI = {10.1007/978-1-4613-9586-7\_3},
       URL = {https://doi.org/10.1007/978-1-4613-9586-7_3},
}

@inproceedings {Osin_survey,
    AUTHOR = {Osin, Denis V.},
     TITLE = {Groups acting acylindrically on hyperbolic spaces},
 BOOKTITLE = {Proceedings of the {I}nternational {C}ongress of
              {M}athematicians---{R}io de {J}aneiro 2018. {V}ol. {II}.
              {I}nvited lectures},
     PAGES = {919--939},
 PUBLISHER = {World Sci. Publ., Hackensack, NJ},
      YEAR = {2018},
      ISBN = {978-981-3272-91-0; 978-981-3272-87-3},
   MRCLASS = {20F65 (20F67 20F69)},
  MRNUMBER = {3966794},
MRREVIEWER = {Camille\ Horbez},
}

@article {Calvez_Euclidean,
    AUTHOR = {Calvez, Matthieu},
     TITLE = {Euclidean {A}rtin-{T}its groups are acylindrically hyperbolic},
   JOURNAL = {Groups Geom. Dyn.},
  FJOURNAL = {Groups, Geometry, and Dynamics},
    VOLUME = {16},
      YEAR = {2022},
    NUMBER = {3},
     PAGES = {963--983},
      ISSN = {1661-7207,1661-7215},
   MRCLASS = {20F36 (20F65 20F67)},
  MRNUMBER = {4506543},
MRREVIEWER = {Valeriy\ G.\ Bardakov},
       DOI = {10.4171/ggd/683},
       URL = {https://doi.org/10.4171/ggd/683},
}

@article {Charney_MW,
    AUTHOR = {Charney, Ruth and Morris-Wright, Rose},
     TITLE = {Artin groups of infinite type: trivial centers and
              acylindrical hyperbolicity},
   JOURNAL = {Proc. Amer. Math. Soc.},
  FJOURNAL = {Proceedings of the American Mathematical Society},
    VOLUME = {147},
      YEAR = {2019},
    NUMBER = {9},
     PAGES = {3675--3689},
      ISSN = {0002-9939,1088-6826},
   MRCLASS = {20F36 (20F65)},
  MRNUMBER = {3993762},
MRREVIEWER = {Sang-hyun\ Kim},
       DOI = {10.1090/proc/14503},
       URL = {https://doi.org/10.1090/proc/14503},
}

@article {CHR,
    AUTHOR = {Ciobanu, Laura and Holt, Derek and Rees, Sarah},
     TITLE = {Equations in groups that are virtually direct products},
   JOURNAL = {J. Algebra},
  FJOURNAL = {Journal of Algebra},
    VOLUME = {545},
      YEAR = {2020},
     PAGES = {88--99},
      ISSN = {0021-8693,1090-266X},
   MRCLASS = {20F10 (20F67 68Q45)},
  MRNUMBER = {4044690},
MRREVIEWER = {Denis\ E.\ Serbin},
       DOI = {10.1016/j.jalgebra.2018.10.044},
       URL = {https://doi.org/10.1016/j.jalgebra.2018.10.044},
}

@article {Vaskou,
    AUTHOR = {Vaskou, Nicolas},
     TITLE = {Acylindrical hyperbolicity for {A}rtin groups of dimension 2},
   JOURNAL = {Geom. Dedicata},
  FJOURNAL = {Geometriae Dedicata},
    VOLUME = {216},
      YEAR = {2022},
    NUMBER = {1},
     PAGES = {Paper No. 7, 28},
      ISSN = {0046-5755,1572-9168},
   MRCLASS = {20F65 (20F36 20F67)},
  MRNUMBER = {4366944},
MRREVIEWER = {Bin\ Sun},
       DOI = {10.1007/s10711-021-00664-5},
       URL = {https://doi.org/10.1007/s10711-021-00664-5},
}

@incollection {Ballman_Buyalo,
    AUTHOR = {Ballmann, Werner and Buyalo, Sergei},
     TITLE = {Periodic rank one geodesics in {H}adamard spaces},
 BOOKTITLE = {Geometric and probabilistic structures in dynamics},
    SERIES = {Contemp. Math.},
    VOLUME = {469},
     PAGES = {19--27},
 PUBLISHER = {Amer. Math. Soc., Providence, RI},
      YEAR = {2008},
      ISBN = {978-0-8218-4286-7},
   MRCLASS = {53C24 (20F67 53C22 53D25)},
  MRNUMBER = {2478464},
MRREVIEWER = {Koichi\ Nagano},
       DOI = {10.1090/conm/469/09159},
       URL = {https://doi.org/10.1090/conm/469/09159},
}

@article {Burger_Mozes,
    AUTHOR = {Burger, Marc and Mozes, Shahar},
     TITLE = {Lattices in product of trees},
   JOURNAL = {Inst. Hautes \'{E}tudes Sci. Publ. Math.},
  FJOURNAL = {Institut des Hautes \'{E}tudes Scientifiques. Publications
              Math\'{e}matiques},
      YEAR = {2000},
     PAGES = {151--194},
      ISSN = {0073-8301,1618-1913},
   MRCLASS = {20E08 (22E40)},
  MRNUMBER = {1839489},
MRREVIEWER = {Alexander\ Lubotzky},
       URL = {http://www.numdam.org/item?id=PMIHES_2000__92__151_0},
}

@phdthesis{Brown-thesis,
    author = {Brown, Tom},
    title = {Uncountably many quasi-isometry classes of groups of type FP via graphical small cancellation theory},
    school={University of Southampton},
    year = {2022},
    month = "1",
}

\end{document}